\documentclass[12pt,reqno,a4paper]{amsart}

\usepackage{fullpage,enumitem,todonotes,framed,nameref,enumitem,csvsimple,amsmath,amssymb,color}
\usepackage[utf8]{inputenc}
\usepackage[T1]{fontenc}
\usepackage[list=true, labelfont=bf, labelformat=brace, position=top]{subcaption}
\definecolor{tuwBlue}{RGB}{0,116,178}
\usepackage[section]{placeins}
\usepackage{hyperref}
\usepackage{stmaryrd}
\usepackage{moreverb}

\title{Plain convergence of adaptive algorithms\\without exploiting reliability and efficiency}

\author{Gregor Gantner}
\author{Dirk Praetorius}

\address{University of Amsterdam, Faculty of Science, Korteweg--de Vries Instituut, Postbus 94248, 1090 GE Amsterdam, The Netherlands}
\email{g.gantner@uva.nl \quad \rm(corresponding author)}

\address{TU Wien, Institute for Analysis and Scientific Computing, 
Wiedner Hauptstra\ss{}e 8-10/E101/4, 1040 Vienna, Austria}
\email{dirk.praetorius@asc.tuwien.ac.at}

\subjclass[2010]{}
\thanks{{\bf Acknowledgement.} The authors thankfully acknowledge support by the Austrian Science Fund (FWF) through the SFB \emph{Taming complexity in partial differential systems}, the stand-alone project \emph{Computational nonlinear PDEs} (grant P33216), and the Erwin Schr\"odinger Fellowship \emph{Optimal adaptivity for space-time methods} (grant J4379).}

\makeatother

\def\refine{{\rm refine}}
\def\set#1#2{\big\{#1 \,:\, #2 \big\}}
\def\d#1{\,{\rm d}#1}

\renewcommand{\H}{\mathsf{H}}

\def\N{\mathbb{N}}
\def\MM{\mathcal{M}}
\def\TT{\mathcal{T}}

\def\R{\mathbb{R}}

\def\PP{\mathcal P}

\def\eps{\varepsilon}

\DeclareMathOperator{\diam}{diam}
\DeclareMathOperator{\dist}{dist}

\DeclareMathOperator{\di}{div}
\renewcommand{\div}{\di}

\def\norm#1#2{|\hspace*{-1.5pt}|#1|\hspace*{-1.5pt}|_{#2}}

\def\conststyle{\sf} 

\def\Crel{C_{\conststyle rel}}
\def\Cred{C_{\conststyle red}}

\def\Cstab{C_{\conststyle stab}}

\usepackage{fancyhdr}
\advance\footskip0.4cm
\advance\textheight-0.4cm
\calclayout
\newcommand*\patchAmsMathEnvironmentForLineno[1]{%
  \expandafter\let\csname old#1\expandafter\endcsname\csname #1\endcsname
  \expandafter\let\csname oldend#1\expandafter\endcsname\csname end#1\endcsname
  \renewenvironment{#1}%
     {\linenomath\csname old#1\endcsname}%
     {\csname oldend#1\endcsname\endlinenomath}}%
\newcommand*\patchBothAmsMathEnvironmentsForLineno[1]{%
  \patchAmsMathEnvironmentForLineno{#1}%
  \patchAmsMathEnvironmentForLineno{#1*}}%
\AtBeginDocument{%
\patchBothAmsMathEnvironmentsForLineno{equation}%
\patchBothAmsMathEnvironmentsForLineno{align}%
\patchBothAmsMathEnvironmentsForLineno{flalign}%
\patchBothAmsMathEnvironmentsForLineno{alignat}%
\patchBothAmsMathEnvironmentsForLineno{gather}%
\patchBothAmsMathEnvironmentsForLineno{multline}%
}
\usepackage[mathlines]{lineno}

\makeatletter
\def\@seccntformat#1{\vspace*{-2mm}\newline\hspace*{4mm}%
  \protect\textup{\protect\@secnumfont
    \ifnum\pdfstrcmp{subsection}{#1}=0 \bfseries\fi
    \csname the#1\endcsname
    \protect\@secnumpunct
  }%
}
\def\paragraph{\@startsection{paragraph}{4}%
  \z@\z@{-\fontdimen2\font}%
  {\normalfont\bfseries}}
\makeatother

\makeatletter
\def\section{\@startsection{section}{1}%
\z@{.7\linespacing\@plus\linespacing}{.5\linespacing}%
{\normalsize\scshape\bfseries\centering}}
\makeatother

\makeatletter
\renewcommand{\@secnumfont}{\bfseries}
\makeatother

\newcounter{statement}
\newenvironment{statement}[2][!]{%
\vskip2mm
\noindent%
\refstepcounter{statement}%
\bf#2~\thestatement%
\ifthenelse{\equal{#1}{!}}{.\ }{~(#1).\ }%
\it%
}{%
\vskip1mm
}

\newenvironment{theorem}[1][!]{\begin{statement}[#1]{Theorem}}{\end{statement}}
\newenvironment{lemma}[1][!]{\begin{statement}[#1]{Lemma}}{\end{statement}}

\newenvironment{proposition}[1][!]{\begin{statement}[#1]{Proposition}}{\end{statement}}

\newenvironment{remark}[1][!]{\begin{statement}[#1]{Remark}}{\end{statement}}
\newenvironment{algorithm}[1][!]{\begin{statement}[#1]{Algorithm}}{\end{statement}}

\def\coarse{H}
\def\fine{h}
\def\XX{\mathcal{X}}
\def\UU{\mathcal{U}}
\def\T{\mathbb{T}}
\def\const#1{C_{\conststyle#1}}
\def\contract#1{q_{\conststyle#1}}
\def\reff#1#2{\stackrel{\eqref{#1}}{#2}}

\begin{document}

\begin{abstract}
We consider $h$-adaptive algorithms in the context of the finite element method (FEM) and the boundary element method (BEM). Under quite general assumptions on the building blocks SOLVE, ESTIMATE, MARK, and REFINE of such algorithms, we prove plain convergence in the sense that the adaptive algorithm drives the underlying {\sl a~posteriori} error estimator to zero. Unlike available results in the literature, our analysis avoids the use of any reliability and efficiency estimate, but only relies on structural properties of the estimator, namely stability on non-refined elements and reduction on refined elements. In particular, the new framework thus covers also problems involving non-local operators like the fractional Laplacian or boundary integral equations, where (discrete) efficiency is (currently) not available.
\end{abstract}

\maketitle
\thispagestyle{fancy}


\def\SS{\mathcal{S}}
\def\XXX{{\color{red}XXX}}
\section{Introduction}

\noindent
{\sl A~posteriori} error estimation and related adaptive mesh-refinement via the loop
\begin{align}\label{eq:semr}
\bf SOLVE 
\quad \longrightarrow \quad
ESTIMATE
\quad \longrightarrow \quad
MARK
\quad \longrightarrow \quad
REFINE
\end{align}
are standard tools in modern scientific computing. Over the last decade, the mathematical understanding has matured. 
Convergence with optimal algebraic rates is mathematically guaranteed for a reasonable class of elliptic model problems and standard discretizations; we refer to the  works~\cite{doerfler1996,mns2000,bdd2004,stevenson2007,ckns2008,cn2012,ffp2014} for some important steps as well as to the state-of-the-art review~\cite{axioms}. However, all these works employ the so-called D\"orfler marking strategy proposed in~\cite{doerfler1996} to single out elements for refinement. Moreover, for the 2D Poisson problem, it has recently been shown that a modified maximum criterion does not only lead to optimal convergence rates, but even leads to instance optimal meshes~\cite{dks2016,ks2016,ip2020+}. As the focus comes to other marking strategies, only plain convergence results are known and the essential works are~\cite{msv2008,siebert2011}. 

To outline the results of~\cite{msv2008,siebert2011} and the contributions of the present work, let us fix some notation.
Let $\XX$ be a normed space which is linked to some domain (or manifold) $\Omega \subset \R^d$, $d \ge 1$. Let $u \in \XX$  be the sought (unknown) solution. Suppose that the discrete subspaces $\XX_\ell \subset \XX$ are linked to some mesh $\TT_\ell$ of $\Omega$ consisting of compact subdomains of $\Omega$. Let $u_\ell \in \XX_\ell$ be a \emph{computable} discrete approximation of $u$. Finally, let $\eta_\ell^2 = \sum_{T \in \TT_\ell} \eta_\ell(T)^2$ be a \emph{computable} error estimator such that $\eta_\ell(T)$ measures, at least heuristically, the error $u - u_\ell$ on $T \in \TT_\ell$. We suppose that the sequence of meshes $(\TT_\ell)_{\ell \in \N_0}$ is generated by the adaptive loop~\eqref{eq:semr}. In such a setting, it has already been observed in the seminal work~\cite{bv1984} that nestedness $\XX_\ell \subseteq \XX_{\ell+1}$ of the discrete spaces together with a C\'ea-type quasi-optimality proves the so-called {\sl a~priori} convergence of adaptive schemes, i.e., there always exists a limit $u_\infty \in \XX$ such that
\begin{align}\label{assumption:limit_u}
 \norm{u_\infty - u_\ell}{\XX} \to 0
 \quad \text{as } \ell \to \infty.
\end{align}
However, it remains to prove that also $u = u_\infty$.

To explain the abstract notation, let us consider the 2D Poisson model problem: In this case,
 $\Omega \subset \R^2$ is a polygonal Lipschitz domain, $f \in L^2(\Omega)$ is some given load, $u \in \XX = H^1_0(\Omega)$ solves the 2D Poisson model problem $-\Delta u = f$ in $\Omega$ subject to homogeneous boundary conditions $u = 0$ on $\partial\Omega$, the meshes $\TT_\ell$ are conforming triangulations of $\Omega$ into compact triangles $T \in \TT_\ell$, and $u_\ell \in \XX_\ell = \set{v_\ell \in H^1_0(\Omega)}{v_\ell|_T \text{ is affine for all } T \in \TT_\ell}$ is the conforming first-order finite element approximation of $u$, which solves
\begin{align}\label{eq:poisson:fem}
 \int_\Omega \nabla u_\ell \cdot \nabla v_\ell \d{x} = \int_\Omega f v_\ell \d{x}
 \quad \text{for all } v_\ell \in \XX_\ell
\end{align}
For this problem, the classical residual error estimator reads
\begin{align}\label{eq:poisson:estimator}
 \eta_\ell^2 = \sum_{T \in \TT_\ell} \eta_\ell(T)^2 
 \quad \text{with} \quad
 \eta_\ell(T)^2 = h_T^2 \, \norm{f}{L^2(T)}^2 + h_T \, \norm{\llbracket\partial_n u_\ell\rrbracket}{L^2(\partial T \cap \Omega)}^2,
\end{align}
where $\llbracket\cdot\rrbracket$ denotes the jump across interior edges and $h_T = |T|^{1/2}$ denotes the local mesh-size; see, e.g., the monographs~\cite{ao00,verfuerth}.

While~\cite{msv2008} formally focusses on conforming Petrov--Galerkin discretizations in the setting of Ladyshenskaja--Babuska--Brezzi (LBB), the actual analysis is more general: Besides some assumptions on the locality of the norms of the involved function spaces, there are no assumptions on how $u$ or $u_\ell$ are computed. The crucial assumptions in~\cite{msv2008} are \emph{local efficiency}
\begin{align}\label{eq:local_efficiency}
 C_{\rm eff}^{-1} \, \eta_\ell(T) 
 \le \norm{u - u_\ell}{\XX(\Omega_\ell(T))} + {\rm osc}_\ell(\Omega_\ell(T))
 \quad \text{for all } T \in \TT_\ell 
\end{align}
as well as \emph{discrete local efficiency} on marked elements 
\begin{align}\label{eq:discrete_local_efficiency}
 C_{\rm eff}^{-1} \, \eta_\ell(T) 
 \le \norm{u_{\ell+1} - u_\ell}{\XX(\Omega_\ell(T))} + {\rm osc}_\ell(\Omega_\ell(T))
 \quad \text{for all } T \in \MM_\ell,
\end{align}
where $\Omega_\ell(T) = \bigcup\set{T' \in \TT_\ell}{T' \cap T \neq \emptyset}$ is the patch of $T$ and ${\rm osc}_\ell$ are some data oscillation terms. It is known that the latter assumption requires (at least) stronger local refinement, e.g., the local bisec5 refinement of marked elements in 2D to ensure the interior node property; see, e.g.,~\cite{mns2000}. The main result of~\cite{msv2008} proves that under these assumptions and for quite general marking strategies (see~\eqref{eq:marking} below), the adaptive algorithm ensures that {\sl a~priori} convergence~\eqref{assumption:limit_u} already implies \emph{estimator convergence}
\begin{align}\label{eq:estimator_convergence}
 \eta_\ell \to 0 
 \quad \text{as } \ell \to \infty.
\end{align}
Provided that the error estimator $\eta_\ell$ additionally satisfies reliability, i.e.,
\begin{align}\label{eq:reliability}
 \norm{u - u_\ell}{\XX} \le \const{rel} \, \eta_\ell,
\end{align}
this proves that $\norm{u - u_\ell}{\XX} \to 0$ as $\ell \to \infty$. In explicit terms, the main result of~\cite{msv2008} reads as follows: If the discrete solutions $u_\ell \in \XX_\ell$ converge~\eqref{assumption:limit_u}, then they converge indeed to the correct limit $u = u_\infty$ --- provided that the error estimator satisfies~\eqref{eq:local_efficiency}--\eqref{eq:discrete_local_efficiency} and~\eqref{eq:reliability}.

Conceptually, it is remarkable that the convergence proof of~\cite{msv2008} exploits lower error bounds, although the mesh-refinement is driven by the error estimator only. The work~\cite{siebert2011} thus aimed to prove convergence without using (discrete) lower bounds. This, however, comes at the cost that, first, the analysis exploits the problem setting (and is restricted to Petrov--Galerkin discretizations of operator equations $Bu = F$) and, second, the ana\-ly\-sis relies on some strengthened reliability estimate (formulated in terms of the residual), which implies~\eqref{eq:reliability}. The main result of~\cite{siebert2011} then states that under these assumptions and for quite general marking strategies (see~\eqref{eq:marking2} below), the adaptive algorithm ensures that {\sl a~priori} convergence~\eqref{assumption:limit_u} already implies \emph{error convergence}
\begin{align}\label{eq:error_convergence}
 \norm{u - u_\ell}{\XX} \to 0
 \quad \text{as } \ell \to \infty.
\end{align}
In particular, the new proof of~\cite{siebert2011} avoids the discrete local efficiency~\eqref{eq:discrete_local_efficiency}.
Surprisingly, however, estimator convergence~\eqref{eq:estimator_convergence} cannot be proved under the assumptions of~\cite{siebert2011} but requires that the error estimator $\eta_\ell$ is also locally efficient~\eqref{eq:local_efficiency}.

One advantage of the results of~\cite{msv2008,siebert2011} is that they apply to  many different {\sl a~posteriori} error estimators. In particular, it has recently been shown in~\cite{fp2020+,gs2020+} that the assumptions of~\cite{siebert2011} are, in particular, satisfied for a wide range of model problems discretized by least squares finite element methods, where adaptivity is driven by the built-in least-squares functional, including even a least-squares space-time discretization of the heat equation. On the other hand,~\cite{siebert2011} excludes adaptive schemes for variational inequalities, and both works~\cite{msv2008,siebert2011} need local efficiency of the error estimator, which does not appear to be available for non-local operators, e.g., finite element methods for the fractional Laplacian (see, e.g.,~\cite{fmp2019+}) or boundary element methods for elliptic integral equations (see, e.g.,~\cite{fkmp2013,gantumur2013}). 

With the latter observations, the current paper comes into play. We provide a new proof for plain convergence of adaptive algorithms, which does \emph{neither} involve reliability~\eqref{eq:reliability} \emph{nor} any kind of (global or local) efficiency~\eqref{eq:local_efficiency}--\eqref{eq:discrete_local_efficiency}. Instead, we exploit that the local contributions of many residual error estimators are weighted by the local mesh-size (cf.~\eqref{eq:poisson:estimator} for the Poisson model problem). With scaling arguments, one usually obtains \emph{reduction on refined elements}
\begin{align}\label{eq:intro:reduction}
 \eta_{\ell+n}(\TT_{\ell+n} \backslash \TT_\ell)^2 
 \le q \, \eta_{\ell}(\TT_\ell \backslash \TT_{\ell+n})^2 
 + C \, \norm{u_{\ell+n} - u_\ell}{\XX}^2
 \quad \text{for all } \ell, n \in \N_0,
\end{align}
with generic constants $0 < q < 1$ and $C > 0$. Note that $\TT_\ell \backslash \TT_{\ell+n}$ corresponds to the elements that are going to be refined, while $\TT_{\ell+n} \backslash \TT_\ell$ corresponds to the generated children. Moreover, there usually holds \emph{stability on non-refined elements}
\begin{align}\label{eq:intro:stability}
 | \eta_{\ell+n}(\TT_{\ell+n} \cap \TT_\ell) - \eta_\ell(\TT_{\ell+n} \cap \TT_\ell) |
 \le C \, \norm{u_{\ell+n} - u_\ell}{\XX}
 \quad \text{for all } \ell,n \in \N_0.
\end{align}
We stress that~\eqref{eq:intro:reduction}--\eqref{eq:intro:stability} play also a fundamental role in the contemporary proofs of optimal convergence rates for adaptive algorithms; see~\cite{axioms}. The main result of the present work (Theorem~\ref{theoerem:estconv:new}) shows that, together with the same marking criterion as in~\cite{msv2008}, the structural properties~\eqref{eq:intro:reduction}--\eqref{eq:intro:stability} suffice to show that {\sl a~priori} convergence~\eqref{assumption:limit_u} yields estimator convergence~\eqref{eq:estimator_convergence}. Clearly, reliability~\eqref{eq:reliability} is then finally required to conclude error convergence~\eqref{eq:error_convergence}. 

\bigskip
\textbf{Outline.} The remainder of this work is organized as follows: In Section~\ref{section:abstract}, we provide a formal statement of the adaptive algorithm (Algorithm~\ref{alg:abstract_algorithm}) as well as precise assumptions on its four modules from~\eqref{eq:semr}. The new plain convergence result (Theorem~\ref{theoerem:estconv:new}, Theorem~\ref{theoerem2:estconv:new}) is stated and proved in Section~\ref{section:main}, before we give some examples which do not fit the framework of~\cite{msv2008,siebert2011}, but are covered by the current analysis.

\section{Abstract adaptive algorithm}
\label{section:abstract}

\subsection{Mesh refinement}\label{section:mesh-refinement}

Let $\Omega \subset \R^d$ be a bounded domain (or a manifold in $\R^d$) with positive measure $|\Omega| > 0$. We say that $\TT_\coarse$ is a mesh (of $\Omega)$, if 
\begin{itemize}
\item $\TT_\coarse$ is a finite set of compact sets $T \in \TT_\coarse$ with positive measure $|T| > 0$;
\item for all $T, T' \in \TT_\coarse$ with $T \neq T'$, it holds that $|T \cap T'| = 0$;
\item $\TT_\coarse$ is a covering of $\overline\Omega$, i.e., $\overline\Omega = \bigcup_{T \in \TT_\coarse} T$.
\end{itemize}
Let $\refine(\cdot)$ be a fixed refinement strategy, i.e., for each mesh $\TT_\coarse$ and a set of marked elements $\MM_\coarse \subseteq \TT_\coarse$, the refinement strategy returns a refined mesh $\TT_\fine := \refine(\TT_\coarse,\MM_\coarse)$ such that, first, at least the marked elements are refined (i.e., $\MM_\coarse \subseteq \TT_\coarse \backslash \TT_\fine$) and, second,
parents $T \in \TT_\coarse$ are the union of their children, i.e.,
\begin{align}\label{eq:union_of_sons}
 T = \bigcup \set{T' \in \TT_\fine}{T' \subseteq T} 
 \quad \text{for all } T \in \TT_\coarse.
\end{align}
For a mesh $\TT_\coarse$, let $\T(\TT_\coarse)$ denote the set of all possible refinements of $\TT_\coarse$ (as determined by the refinement strategy $\refine(\cdot)$), i.e., for any $\TT_\fine \in \T(\TT_\coarse)$, there exists $n \in \N_0$ and $\TT'_0, \dots, \TT'_n$ such that $\TT'_0 = \TT_\coarse$, $\TT'_{j+1} = \refine(\TT'_j,\MM'_j)$ for all $j = 0, \dots, n-1$ and appropriate $\MM'_j \subseteq \TT'_j$, and $\TT'_n = \TT_\fine$.
Finally, we suppose that we are given a fixed initial mesh $\TT_0$ so that it makes sense to call $\T := \T(\TT_0)$ the \emph{set of all admissible meshes.}

\subsection{Continuous and discrete setting}

Let $\XX$ be a normed space (related to $\Omega$) and $u \in \XX$ be the (unknown) exact solution. For each mesh $\TT_\coarse$, let $\XX_\coarse \subseteq \XX$ be an associated discrete subspace and $u_\coarse \in \XX_\coarse$ be the corresponding (computable) discrete solution.

\subsection{Error estimator}

For each mesh $\TT_\coarse$ and all $T \in \TT_\coarse$, let $\eta_\coarse(T) \ge 0$ be a \emph{computable} quantity which is usually called \emph{refinement indicator}. At least heuristically, $\eta_\coarse(T)$ measures the error $u - u_\coarse$ on the element $T$. We abbreviate 
\begin{align}
 \eta_\coarse := \eta_\coarse(\TT_\coarse),
 \quad \text{where} \quad
 \eta_\coarse(\UU_\coarse) := \Big( \sum_{T \in \TT_\coarse} \eta_\coarse(T)^2 \Big)^{1/2}
 \quad \text{for all } \UU_\coarse \subseteq \TT_\coarse.
\end{align}
We note that $\eta_\coarse$ is usually referred to as \emph{error estimator}. 

\subsection{Adaptive algorithm}

Starting from the given initial mesh $\TT_0$, we consider the standard adaptive loop~\eqref{eq:semr} in the following algorithmic form:

\begin{algorithm}
\label{alg:abstract_algorithm}
For each $\ell=0,1,2,\dots$, iterate the following steps~{\rm(i)--(iv)}:
\begin{itemize}
\item[\rm(i)] {\bf SOLVE:} Compute the discrete solution $u_\ell \in \XX_\ell$. 
\item[\rm(ii)] {\bf ESTIMATE:} Compute  refinement indicators $\eta_\ell(T)$ for all elements ${T}\in\TT_\ell$. 
\item[\rm(iii)] {\bf MARK:} Determine a set of marked elements $\MM_\ell\subseteq\TT_\ell$.
\item[\rm(iv)] {\bf REFINE:} Generate the refined mesh $\TT_{\ell+1}:=\refine(\TT_\ell,\MM_\ell)$. 
\end{itemize}
\textbf{Output:} Refined meshes $\TT_\ell$, corresponding exact discrete solutions $u_\ell$, and 
error estimators $\eta_\ell$ for all $\ell \in \N_0$.\qed
\end{algorithm}

To analyze Algorithm~\ref{alg:abstract_algorithm}, it remains to specify further assumptions on its four modules: As far as {\bf SOLVE} is concerned, we shall only assume {\sl a~priori} convergence~\eqref{assumption:limit_u}.
While this assumption is guaranteed for many problems (see, e.g.,~\cite{msv2008,siebert2011} for problems in the framework of the LBB theory as well as the seminal work~\cite{bv1984} for problems in the Lax--Milgram setting), we stress that, at this point, it is still mathematically unclear whether there holds $u_\infty = u$ or not. 

As far as {\bf MARK} is concerned, let $M \colon \R_{\ge0} \to \R_{\ge0}$ be continuous at $0$ with $M(0)=0$ and suppose that the sets $\MM_\ell \subseteq \TT_\ell$ satisfy the following property from~\cite{msv2008}:
\begin{align}\label{eq:marking}
 \max_{T\in\TT_\ell\setminus\MM_\ell} \eta_\ell(T) \le M\big(\eta_\ell(\MM_\ell)\big).
\end{align}
We note that the latter assumption is weaker than the following assumption from~\cite{siebert2011}:
\begin{align}\label{eq:marking2}
 \max_{T\in\TT_\ell\setminus\MM_\ell} \eta_\ell(T) \le M\big(\max_{T \in \MM_\ell}\eta_\ell(T)\big).
\end{align}
Clearly, the marking criteria~\eqref{eq:marking}--\eqref{eq:marking2} are satisfied with $M(t) = t$ as soon as $\MM_\ell$ contains one element with maximal indicator, i.e., there exists $T \in \MM_\ell$ such that $\eta_\ell(T) = \max_{T' \in \TT_\ell} \eta_\ell(T')$. For instance, this is the case for
\begin{itemize}
\item the \emph{maximum criterion} for some fixed $0 \le \theta \le 1$, where 
\begin{align}\label{eq:marking:maximum}
 \MM_\ell:=\set{T\in\TT_\ell}{\eta_\ell(T) \ge (1-\theta) \max_{T'\in\TT_\ell} \eta_\ell(T')}
\end{align}
\item the \emph{equidistribution criterion} for fixed $0 \le \theta \le 1$, where
\begin{align}\label{eq:marking:equilibration}
 \MM_\ell := \set{T\in\TT_\ell}{\eta_\ell(T)\ge (1-\theta)\, \eta_\ell/\#\TT_\ell}
\end{align}
\end{itemize}
Finally, let us consider the \emph{D\"orfler criterion} for some fixed $0<\theta\le1$ , i.e., 
\begin{align}\label{eq:marking:doerfler}
 \theta\,\eta_\ell^2 \le \eta_\ell(\MM_\ell)^2.
\end{align}
While~\eqref{eq:marking2} cannot be satisfied in general,~\eqref{eq:marking} holds with $M(t):=\sqrt{(1-\theta)\,\theta^{-1}t}$.
To see this, let $T\in\TT_\ell\setminus\MM_\ell$ and note that
\begin{align*}
 \eta_\ell(T)^2 
 &\le \eta_\ell(\TT_\ell\setminus\MM_\ell)^2 
 = \eta_\ell^2 - \eta_\ell(\MM_\ell)^2 \le (1-\theta)\eta_\ell^2 \le (1-\theta)\,\theta^{-1} \eta_\ell(\MM_\ell)^2.
\end{align*}
However, if the set $\MM_\ell$ is constructed via sorting of the indicators, then
\begin{align}\label{eq:marking:doerfler:max}
 \max_{T'\in\TT_\ell\setminus\MM_\ell} \eta_\ell(T') 
 \le \min_{T'\in\MM_\ell} \eta_\ell(T');
\end{align}
see \cite{pp2020} for different algorithms which generate $\MM_\ell \subseteq \TT_\ell$ satisfying the D\"orfler criterion~\eqref{eq:marking:doerfler} together with~\eqref{eq:marking:doerfler:max}. In the latter case,~\eqref{eq:marking}--\eqref{eq:marking2} hold again with $M(t):=t$.
%

\section{A new plain convergence result}
\label{section:main}

Unlike~\cite{msv2008,siebert2011}, we only require the following two structural properties of the error estimator for all $\TT_\coarse \in \T$ and all refinements $\TT_\fine \in \T(\TT_\coarse)$, where $S, R \colon \R_{\ge0} \to \R_{\ge0}$ are functions which are continuous at $0$ with $R(0) = 0 = S(0)$ and $0 < \contract{red} < 1$:
\begin{itemize}
\item {\bf stability on non-refined elements}, i.e., 
\begin{align}\label{eq:axiom:stability}
 \eta_\fine(\TT_\fine \cap \TT_\coarse) 
 \le \eta_\coarse(\TT_\fine \cap \TT_\coarse)
 + S(\norm{u_\fine - u_\coarse}{\XX});
\end{align}
\item {\bf reduction on refined elements}, i.e., 
\begin{align}\label{eq:axiom:reduction}
 \eta_\fine(\TT_\fine \backslash \TT_\coarse)^2
 \le 
 \contract{red} \, \eta_\coarse(\TT_\coarse \backslash \TT_\fine)^2
 + R(\norm{u_\fine - u_\coarse}{\XX}).
\end{align}
\end{itemize}
We note that, \eqref{eq:axiom:stability}--\eqref{eq:axiom:reduction} are implicitly first found in the proof of~\cite[Corollary~3.4]{ckns2008}, but already seem to go back to~\cite{dk2008} (used there for the oscillations).
In practice, the reduction~\eqref{eq:axiom:reduction} can only be proved if the local contributions $\eta_\coarse(T)$ of the error estimator are weighted by (some positive power of) the local mesh-size $h_T$. In the later examples in Section~\ref{section:examples}, it holds that $S(t) \sim t$ and $R(t) \sim t^2$.

Under the structural assumptions~\eqref{eq:axiom:stability}--\eqref{eq:axiom:reduction} on the estimator, the following theorem already proves that Algorithm~\ref{alg:abstract_algorithm} leads to estimator convergence. We stress that neither the reliability estimate~\eqref{eq:reliability} nor any (global or even local) efficiency estimate (e.g.,~\eqref{eq:local_efficiency}--\eqref{eq:discrete_local_efficiency}) is required.

\begin{theorem}\label{theoerem:estconv:new}
Suppose the properties~\eqref{eq:axiom:stability}--\eqref{eq:axiom:reduction} of the estimator and that refinement ensures that each parent is the union of its children~\eqref{eq:union_of_sons}. Consider the output of Algorithm~\ref{alg:abstract_algorithm}
with the marking strategy~\eqref{eq:marking}.
Then, {\sl a~priori} convergence~\eqref{assumption:limit_u} implies estimator convergence
\begin{align}\label{eq:theoerem:estconv:new}
 \eta_\ell \to 0 \quad \text{as } \ell \to \infty.
\end{align}
\end{theorem}

The proof of Theorem~\ref{theoerem:estconv:new} employs the following elementary result, whose simple proof is included for the convenience of the reader.

\begin{lemma}\label{axioms:lemma:cor:estimator_convergence}
Let $(a_\ell)_{\ell \in \N_0}$ be a sequence with $a_\ell \ge 0$ for all $\ell \in \N_0$.
Suppose that there exists $0 < \rho < 1$ and a sequence $(b_\ell)_{\ell \in \N_0}$ with $b_\ell \to 0$ as $\ell \to \infty$ such that
\begin{align}
 a_{\ell+1} \le \rho a_\ell + b_\ell
 \quad \text{for all } \ell \in \N_0.
\end{align}
Then, if follows that $a_\ell \to 0$ as $\ell \to \infty$.
\end{lemma}

\begin{proof}
With the convergence of $(b_\ell)_{\ell \in \N_0}$, we note that
\begin{align*}
0 \le \limsup_{\ell\to\infty} a_\ell
= \limsup_{\ell\to\infty} a_{\ell+1}
&\le \limsup_{\ell\to\infty} \, ( \rho a_\ell + b_\ell ) = \rho \, \limsup_{\ell\to\infty}a_\ell.
\end{align*}
Thus, it only remains to show that $\limsup_{\ell\to\infty}a_\ell < \infty$ to conclude that $0=\liminf_{\ell\to\infty} a_\ell = \limsup_{\ell\to\infty}a_\ell$ and hence $\lim_{\ell\to\infty} a_\ell = 0$.
Indeed, induction on $\ell$ proves  that
\begin{align*}
0 \le a_{\ell}
&\le \rho^{\ell} a_0 + \sum_{j=0}^{\ell-1} \rho^{\ell-1-j} b_j
\quad \text{for all $\ell\in\N_0$.}
\end{align*}
Since $(b_\ell)_{\ell \in \N_0}$ is uniformly bounded, the geometric series yields that $\sup_{\ell \in \N} a_\ell < \infty$. 
In particular, we thus see that $\limsup_{\ell\to\infty}a_\ell < \infty$ and conclude the proof.
\end{proof}

\def\eps{\varepsilon}
\begin{proof}[Proof of Theorem~\ref{theoerem:estconv:new}]
The proof is split into five steps. 

{\bf Step 1:} We prove that $(\TT_\ell)_{\ell \in \N_0}$ admits a subsequence $(\TT_{\ell_k})_{k \in \N_0}$ such that 
$\eta_{\ell_k}(\TT_{\ell_k} \backslash \TT_{\ell_{k+1}}) \to 0$ as $k \to \infty$.
Let $\TT_\infty = \bigcup_{\ell' \in \N_0} \bigcap_{\ell \ge \ell'} \TT_\ell$ be the set of all elements, which remain unrefined after some (arbitrary) step $\ell'$. We exploit~\eqref{eq:union_of_sons} and choose a subsequence $(\TT_{\ell_k})_{k \in \N_0}$ of $(\TT_\ell)_{\ell \in \N_0}$ such that 
\begin{align}\label{eq10:theoerem:estconv:new}
 \TT_{\ell_{k+1}} \cap \TT_{\ell_k} = \TT_{\ell_k} \cap \TT_\infty
 \quad \text{for all } k \in \N_0,
\end{align}
i.e., only elements $T \in \TT_{\ell_k} \cap \TT_\infty$ remain unrefined if we pass from $\TT_{\ell_k}$ to $\TT_{\ell_{k+1}}$.
Note that the choice of $(\TT_{\ell_k})_{k \in \N_0}$ guarantees the inclusion
\begin{align*}
 \TT_{\ell_{k+1}} \cap \TT_{\ell_k} 
 \reff{eq10:theoerem:estconv:new}= \TT_{\ell_k} \cap \TT_\infty 
 \subseteq \TT_{\ell_{k+1}} \cap \TT_\infty 
 \reff{eq10:theoerem:estconv:new}= \TT_{\ell_{k+2}} \cap \TT_{\ell_{k+1}} 
 \end{align*}
 and hence 
\begin{align*}
 \TT_{\ell_{k+1}} \backslash \TT_{\ell_{k+2}}
 = \TT_{\ell_{k+1}} \backslash [ \TT_{\ell_{k+2}} \cap \TT_{\ell_{k+1}} ]
 \subseteq \TT_{\ell_{k+1}} \backslash [ \TT_{\ell_{k+1}} \cap \TT_{\ell_k} ]
 = \TT_{\ell_{k+1}} \backslash \TT_{\ell_k}.
\end{align*}
With this and reduction~\eqref{eq:axiom:reduction}, we infer that
\begin{align*}
 \eta_{\ell_{k+1}}(\TT_{\ell_{k+1}} \backslash \TT_{\ell_{k+2}})^2
 &\le \eta_{\ell_{k+1}}(\TT_{\ell_{k+1}} \backslash \TT_{\ell_k})^2
 \reff{eq:axiom:reduction}\le \contract{red} \, \eta_{\ell_k}(\TT_{\ell_k} \backslash \TT_{\ell_{k+1}})^2
 + R(\norm{u_{\ell_{k+1}} - u_{\ell_k}}{\XX}).
\end{align*}
With $a_k = \eta_{\ell_k}(\TT_{\ell_k} \backslash \TT_{\ell_{k+1}})^2$, $\rho := \contract{red}$, and $b_k := R(\norm{u_{\ell_{k+1}} - u_{\ell_k}}{\XX})$, {\sl a~priori} convergence~\eqref{assumption:limit_u} proves that
\begin{align*}
 0 \le a_{k+1} \le \rho \, a_k + b_k
 \quad \text{for all } k \in \N_0
 \quad \text{with } \lim_{k \to \infty} b_k = 0.
\end{align*}
By use of Lemma~\ref{axioms:lemma:cor:estimator_convergence}, we conclude that $a_k = \eta_{\ell_k}(\TT_{\ell_k} \backslash \TT_{\ell_{k+1}})^2 \to 0$ as $k \to \infty$.

{\bf Step 2:} We prove that the subsequence $(\TT_{\ell_k})_{k \in \N_0}$ also guarantees that $\eta_{\ell_k}(\MM_{\ell_k}) \to 0$ as $k \to \infty$. To this end, we first note that $T \in \TT_{\ell_k} \backslash \TT_{\ell_k+1}$ implies that $T$ is refined and hence $T \not \in \TT_\infty$. Therefore, we see that
\begin{align*}
 \MM_{\ell_k} \subseteq \TT_{\ell_k} \backslash \TT_{\ell_k+1}
 \subseteq \TT_{\ell_k} \backslash \TT_\infty
 = \TT_{\ell_k} \backslash [\TT_{\ell_k} \cap \TT_\infty ]
 \reff{eq10:theoerem:estconv:new}= \TT_{\ell_k} \backslash [\TT_{\ell_{k+1}} \cap \TT_{\ell_k}]
 = \TT_{\ell_k} \backslash \TT_{\ell_{k+1}}.
\end{align*}
This implies that $0 \le \eta_{\ell_k}(\MM_{\ell_k}) \le \eta_{\ell_k}(\TT_{\ell_k} \backslash \TT_{\ell_{k+1}}) \to 0$ 
as $k \to \infty$.

{\bf Step 3:} For all fixed $\ell' \in \N_0$, we prove that
\begin{align}\label{eq:lemma70:new}
 \eta_{\ell_k}(\TT_{\ell'} \cap \TT_\infty) \to 0
 \quad \text{as } \ell' \le \ell_k \to \infty \text{ \  together with \ } k \to \infty.
\end{align}
To see this, we exploit the marking strategy~\eqref{eq:marking} and note with Step~2 that
\begin{align*}
 \max_{T \in \TT_{\ell_k} \backslash \MM_{\ell_k}} \eta_{\ell_k}(T)
 \le M \big( \eta_{\ell_k}(\MM_{\ell_k}) \big) \to 0
 \quad \text{as } k \to \infty.
\end{align*}
For $\ell' \le \ell_k$, it holds that $\TT_{\ell'} \cap \TT_\infty \subseteq \TT_{\ell_k} \backslash \MM_{\ell_k}$ and hence
\begin{align*}
 \eta_{\ell_k}(T) \to 0 \quad \text{as } k \to \infty
 \quad \text{for all } T \in \TT_{\ell'} \cap \TT_\infty.
\end{align*}
Since $\TT_{\ell'} \cap \TT_\infty \subseteq \TT_{\ell_k}$ is a fixed finite set, we conclude the proof of~\eqref{eq:lemma70:new}.

{\bf Step 4:} We prove that $(\TT_{\ell_k})_{k \in \N_0}$ admits a subsequence $(\TT_{\ell_{k_j}})_{j \in \N_0}$ such that $\eta_{\ell_{k_j}} \to 0$ as $j \to \infty$. To this end, let first $\TT_\coarse \in \T$ and $\TT_\fine \in \T(\TT_\coarse)$.
Reduction~\eqref{eq:axiom:reduction}
proves that
\begin{subequations}\label{theoreom:new:eq1}
\begin{align}\nonumber
 \eta_\fine^2 
 &=  
 \eta_\fine(\TT_\fine \backslash \TT_\coarse)^2 + \eta_\fine(\TT_\fine \cap \TT_\coarse)^2 
 \\& 
 \le
 \contract{red} \, \eta_\coarse(\TT_\coarse \backslash \TT_\fine)^2 + \eta_\fine(\TT_\fine \cap \TT_\coarse)^2 
 + R(\norm{u_\fine - u_\coarse}{\XX})
 \\& \nonumber
 \le \contract{red} \, \eta_\coarse^2 + \eta_\fine(\TT_\fine \cap \TT_\coarse)^2 
 + R(\norm{u_\fine - u_\coarse}{\XX}).
\intertext{If $\eta_\coarse = 0$, stability~\eqref{eq:axiom:stability} and reduction~\eqref{eq:axiom:reduction} prove that}
 \eta_\fine^2 \nonumber
 &= 
 \eta_\fine(\TT_\fine \backslash \TT_\coarse)^2 + \eta_\fine(\TT_\fine \cap \TT_\coarse)^2
 \\& 
 \le \contract{red} \, \eta_\coarse(\TT_\coarse \backslash \TT_\fine)^2
 + 2 \, \eta_\coarse(\TT_\coarse \cap \TT_\fine)^2
 + [R + 2S^2](\norm{u_\fine - u_\coarse}{\XX})
 \\& \nonumber
 =  [R + 2S^2](\norm{u_\fine - u_\coarse}{\XX}),
\end{align}
\end{subequations}
where $[R+ 2 S^2](t) = R(t) + 2 S(t)^2$.
For the subsequence~$(\TT_{\ell_k})_{k \in \N_0}$, we recall from~\eqref{eq10:theoerem:estconv:new} that $\TT_{\ell_{k+n}} \cap \TT_{\ell_k} = \TT_{\ell_k} \cap \TT_\infty$. Hence,
the estimates~\eqref{theoreom:new:eq1} read, for all $k\in \N_0$ and $n \in \N$,
\begin{align*}
 \eta_{\ell_{k+n}}^2
 &\le  \contract{red} \, \eta_{\ell_k}^2 + \eta_{\ell_{k+n}}(\TT_{\ell_k} \cap \TT_\infty)^2
 + R(\norm{u_{\ell_{k+n}} - u_{\ell_k}}{\XX}),
 \quad \text{if $\eta_{\ell_k}^2 \neq 0$},
\intertext{resp.}
 \eta_{\ell_{k+n}}^2
 &\le [R+2S^2](\norm{u_{\ell_{k+n}} - u_{\ell_k}}{\XX}),
 \quad \text{if $\eta_{\ell_k}^2 = 0$}.
\end{align*}
Let $0 < \contract{red} < \contract{red}' < 1$. Given $k \in \N_0$ with $\eta_{\ell_k}^2 \neq 0$, the convergence~\eqref{eq:lemma70:new}
allows us to pick some $n(k) \in \N$ such that 
\begin{align*}
 \contract{red} \, \eta_{\ell_k}^2 + \eta_{\ell_{k+n(k)}}(\TT_{\ell_k} \cap \TT_\infty)^2
 \le  \contract{red}' \, \eta_{\ell_k}^2.
\end{align*}
In particular, we can choose a further subsequence $(\TT_{\ell_{k_j}})_{j \in \N_0}$ of $(\TT_{\ell_k})_{k \in \N_0}$ such that
\begin{align*}
 \eta_{\ell_{k_{j+1}}}^2
 \le \contract{red}' \, \eta_{\ell_{k_j}}^2 + [R+2S^2](\norm{u_{\ell_{k_{j+1}}} - u_{\ell_{k_j}}}{\XX})
 \quad \text{for all } j \in \N_0.
\end{align*}
With $a_j = \eta_{\ell_{k_j}}^2$, $\rho = \contract{red}'$, and $b_j = [R+2S^2](\norm{u_{\ell_{k_{j+1}}} - u_{\ell_{k_j}}}{\XX})$, {\sl a~priori} convergence~\eqref{assumption:limit_u} proves that
\begin{align*}
 0 \le a_{j+1} \le \rho \, a_j + b_j
 \quad \text{for all } j \in \N_0
 \quad \text{with } \lim_{j \to \infty} b_j = 0.
\end{align*}
By use of Lemma~\ref{axioms:lemma:cor:estimator_convergence}, we see that $a_j = \eta_{\ell_{k_j}}^2 \to 0$ as $j \to \infty$. 

{\bf Step~5:} We prove that convergence of the subsequence $\eta_{\ell_{k_j}} \to 0$ already implies convergence of the full sequence $\eta_\ell \to 0$ as $\ell \to \infty$. To this end, we argue as in~\eqref{theoreom:new:eq1} and use $\contract{red} \le 2$ to see that, for all $\TT_\coarse \in \T \text{ and all } \TT_\fine \in \T(\TT_\coarse)$,
\begin{align}\label{theoreom:new:eq2}
 \eta_\fine^2 
 \le 2 \, \eta_\coarse^2 + [R+2S^2] (\norm{u_\fine - u_\coarse}{\XX}).
\end{align}
Given $\eps > 0$, there exists an index $\ell_{k_j}$ such that $\eta_{\ell_{k_j}}\le \eps$ and $[R+2S^2](\norm{u_\ell - u_{\ell_{k_j}}}{\XX}) \le \eps$ for all $\ell \ge \ell_{k_j}$. For $\ell \ge \ell_{k_j}$, estimate~\eqref{theoreom:new:eq2} thus proves that
\begin{align*}
 \eta_\ell^2 \le 2\eps^2 + \eps.
\end{align*}
This concludes the proof.
\end{proof}

For the D\"orfler marking criterion~\eqref{eq:marking:doerfler}, the 
refinement assumption~\eqref{eq:union_of_sons} exploited in the proof of Theorem~\ref{theoerem:estconv:new} can even be dropped. The following result (together with its very simple proof) is essentially the key argument in~\cite{afp2012}.

\begin{theorem}\label{theoerem2:estconv:new}
Suppose the properties~\eqref{eq:axiom:stability}--\eqref{eq:axiom:reduction} of the estimator. Consider the output of Algorithm~\ref{alg:abstract_algorithm}, where the marked elements $\MM_\ell \subseteq \TT_\ell$ satisfy the D\"orfler criterion~\eqref{eq:marking:doerfler} for some fixed marking parameter $0 < \theta \le 1$. Then, {\sl a~priori} convergence~\eqref{assumption:limit_u} yields estimator convergence~\eqref{eq:theoerem:estconv:new}.
\end{theorem}

\begin{proof}
Let $\TT_\coarse \in \T$ and $\TT_\fine \in \T(\TT_\coarse)$. Arguing as for~\eqref{theoreom:new:eq1} and exploiting the Young inequality for some arbitrary $\delta > 0$ (instead of $\delta = 1$ above), we see that
\begin{align*}
 &\eta_\fine^2 
 \le
 \contract{red} \, \eta_\coarse(\TT_\coarse \backslash \TT_\fine)^2 + (1+\delta) \, \eta_\coarse(\TT_\fine \cap \TT_\coarse)^2 
 + [R + (1+\delta^{-1}) S^2] (\norm{u_\fine - u_\coarse}{\XX})
 \\& \quad
 = (1+\delta) \, \eta_\coarse^2 - [(1+\delta) -\contract{red}] \, \eta_\coarse(\TT_\coarse \backslash \TT_\fine)^2 
 + [R + (1+\delta^{-1}) S^2] (\norm{u_\fine - u_\coarse}{\XX}).
\end{align*}
For the sequence $(\TT_\ell)_{\ell \in \N_0}$, the D\"orfler criterion~\eqref{eq:marking:doerfler} and $\MM_\ell\subseteq\TT_\ell\setminus\TT_{\ell+1}$ yield that
\begin{align*}
 \theta \eta_\ell^2 \le \eta_\ell(\MM_\ell)^2 \le \eta_\ell(\TT_\ell \backslash \TT_{\ell+1})^2.
\end{align*}
Combining the latter estimates, we obtain that
\begin{align*}
 \eta_{\ell+1}^2
 &\,\,\le\,\, (1+\delta) \, \eta_\ell^2 - [(1+\delta) -\contract{red}] \, \eta_\ell(\TT_\ell \backslash \TT_{\ell+1})^2 
 +[R + (1+\delta^{-1}) S^2] (\norm{u_{\ell+1} - u_\ell}{\XX}) \\
 &\reff{eq:marking:doerfler}\le \big((1+\delta) - [(1+\delta) -\contract{red}] \, \theta \big) \, \eta_\ell^2 
 + [R + (1+\delta^{-1}) S^2] (\norm{u_{\ell+1} - u_\ell}{\XX}). 
\end{align*}
We define $\rho = (1+\delta) - [(1+\delta) -\contract{red}] \, \theta > 0$ as well as  $a_\ell = \eta_\ell^2$ and $b_\ell = [R + (1+\delta^{-1}) S^2] (\norm{u_{\ell+1} - u_\ell}{\XX})$. Choosing $\delta > 0$ sufficiently small, we observe that $0 < \rho < 1$ and that
{\sl a~priori} convergence~\eqref{assumption:limit_u} proves that
\begin{align*}
 0 \le a_{\ell + 1} \le \rho \, a_\ell + b_\ell
 \quad \text{for all } \ell \in \N_0
 \quad \text{with } \lim_{\ell \to \infty} b_\ell = 0.
\end{align*}
By use of Lemma~\ref{axioms:lemma:cor:estimator_convergence}, we conclude that $a_\ell = \eta_\ell^2 \to 0$ as $\ell\to \infty$. 
\end{proof}

\begin{remark}
The proof of Theorem~\ref{theoerem2:estconv:new} shows that in case of the D\"orfler marking~\eqref{eq:marking:doerfler}, it is sufficient to have stability~\eqref{eq:axiom:stability} and reduction~\eqref{eq:axiom:reduction} for one-level refinements, i.e., \eqref{eq:axiom:stability}--\eqref{eq:axiom:reduction} are only required for all $\TT_\coarse \in \T$, all $\MM_\coarse \subseteq \TT_\coarse$, and $\TT_\fine = \refine(\TT_\coarse, \MM_\coarse)$. 
\end{remark}%

\section{Examples}
\label{section:examples}

\def\d#1{{\,\rm d}#1}
\def\H{\widetilde{H}}
\def\AA{\mathcal{A}}
\def\EE{\mathcal{E}}
\subsection{Laplace obstacle problem}

Let $\Omega \subset \R^d$, $d \ge 2$, be a bounded Lipschitz domain. Let $\chi$ be an affine function
with $\chi \le 0$ on $\partial\Omega$. Denote the set of admissible functions by
\begin{align}
 \AA := \set{v \in \XX}{v \ge \chi \text{ a.e.\ on } \Omega} \quad\text{with }\XX:=H^1_0(\Omega)
\end{align}
and note that $\AA \neq \emptyset$.
The minimization problem reads: Given a continuous linear functional $f \in H^{-1}(\Omega)$ with , find $u \in \AA$ such that
\begin{align}\label{eq:obstacle:strong}
 E(u) = \min_{v \in \AA} E(v),
 \quad \text{where }
 E(v) := \frac{1}{2} \norm{\nabla v}{L^2(\Omega)}^2 - \int_\Omega f v \d{x}.
\end{align}
Since $\AA \neq \emptyset$ is convex and closed, it is well-known~\cite[Theorem~II.2.1]{MR1786735} that the minimization problem~\eqref{eq:obstacle:strong} admits a unique solution $u \in \AA$. 

We consider regular triangulations $\TT_\coarse$ of $\Omega$ into non-degenerate compact simplices and the corresponding first-order Courant finite element space 
\begin{align}\label{eq:courant}
 \XX_\coarse := \set{v_\coarse \in H^1_0(\Omega)}{v_\coarse|_T \text{ is affine for all } T \in \TT_\coarse}. 
\end{align}
Then, $\AA_\coarse := \AA \cap \XX_\coarse \neq \emptyset$ is closed and convex. As in the continuous case, there hence exists a unique discrete minimizer $u_\coarse \in \AA_\coarse$ such that
\begin{align}
 E(u_\coarse) = \min_{v_\coarse \in \AA_\coarse} E(v_\coarse).
\end{align}
Under additional regularity $f \in L^2(\Omega)$, one can argue as in~\cite{bch2007} to show reliability
\begin{align}\label{eq:obstacle:reliable}
 \frac{1}{2} \norm{\nabla(u - u_\coarse)}{L^2(\Omega)}^2
 \le E(u_\coarse) - E(u) 
 \le \Crel^2 \eta_\coarse^2,
\end{align}
for the residual error estimator with local contributions
\begin{align}\label{eq:obstacle:estimator}
 \eta_\coarse(T)^2 = h_T \norm{\llbracket \partial_n u_\coarse \rrbracket}{L^2(\partial T\cap\Omega)}^2
 + h_T^2 \sum_{\substack{E \in \EE_\coarse(T) \\ E \subset \partial\Omega}} \norm{f}{L^2(T)}^2
 + h_T^2 \sum_{\substack{E \in \EE_\coarse(T) \\ E \not\subset\partial\Omega}} \norm{f - f_E}{L^2(T)}^2,
\end{align}
where $h_T := |T|^{1/d}$ and $\EE_\coarse(T)$ is the set of all $(d-1)$-dimensional facets (i.e., edges for $d = 2$) and $f_E \in \R$ is the integral mean of $f$ over the corresponding patch $\omega_E = T \cup T'$ with $E = T \cap T'$. In addition to~\eqref{eq:union_of_sons}, we suppose uniform contraction of the mesh-size on refined elements, i.e., there exists $0 < \contract{ctr} < 1$
\begin{align}\label{eq:reduction:h}
 |T'| \le \contract{ctr} |T|
 \quad \text{for all \ } T \in \TT_\coarse \in \T \text{ \ and all \ } T' \in \TT_\fine \in \T(\TT_\coarse) \text{ \ with \ } T' \subsetneqq T.
\end{align}
Then, stability~\eqref{eq:axiom:stability} and reduction~\eqref{eq:axiom:reduction} with $S(t) = \Cstab \,t$ and $R(t) = \Cred \, t^2$ follow as for the linear case; see~\cite{ckns2008} (or~\cite{pp2011} for the obstacle problem). The constants $\Crel, \Cstab > 0$ depend only on $\Omega$, $d$, and uniform $\gamma$-shape regularity of the admissible meshes $\TT_\coarse \in \T$ in the sense of 
\begin{align}\label{eq:gamma-shape-regular}
 \gamma := \sup_{\TT_\coarse \in \T} \max_{T \in \TT_\coarse} \frac{{\rm diam}(T)}{|T|^{1/d}} < \infty,
\end{align}
while $\Cred > 0$ and $0 < \contract{red} < 1$ depend additionally on $\contract{ctr}$.
The {\sl a~priori} convergence~\eqref{assumption:limit_u} follows essentially as in the seminal work~\cite{bv1984}: The assumption~\eqref{eq:union_of_sons} on the mesh-refinement implies that refinement leads to nested spaces, i.e., Algorithm~\ref{alg:abstract_algorithm} leads to $\XX_\ell \subseteq \XX_{\ell+1}$ and hence $\AA_\ell \subseteq \AA_{\ell+1}$ for all $\ell \in \N_0$. Therefore, $\AA_\infty := {\rm closure}\big(\bigcup_{\ell \in \N_0} \AA_\ell\big) \neq \emptyset$ is a closed and convex subset of $\XX$ and thus gives rise to a unique minimizer $u_\infty \in \AA_\infty$ such that
\begin{align}
 E(u_\infty) = \min_{v_\infty \in \AA_\infty} E(v_\infty).
\end{align}
Based on estimates for the equivalent variational inequalities (see~\cite[Theorem~II.2.1]{MR1786735}), it follows that
\begin{align}
 \norm{u_\infty - u_\ell}{\XX}^2
 \lesssim \inf_{v_\ell \in \AA_\ell}
 \norm{u_\infty - v_\ell}{\XX}
 \to 0 \quad \text{as } \ell \to \infty,
\end{align}
where we stress the different powers of the norms which are due to lack of Galerkin orthogonality; see, e.g.,~\cite{falk1974}. Overall, we thus get the following plain convergence result, where we note that for D\"orfler marking~\eqref{eq:marking:doerfler} with sufficiently small marking parameter $0 < \theta \ll 1$, \cite{ch2015} proves even rate-optimal convergence for $d = 2$.

\begin{proposition}
As long as the mesh-refinement strategy guarantees regular simplicial triangulations satisfying~\eqref{eq:union_of_sons}, \eqref{eq:reduction:h}, and~\eqref{eq:gamma-shape-regular}
 and as long as the marking strategy satisfies~\eqref{eq:marking}, Algorithm~\ref{alg:abstract_algorithm} for the Laplace obstacle problem~\eqref{eq:obstacle:strong} driven by the indicators~\eqref{eq:obstacle:estimator} yields convergence
\begin{align}\tag*{\qed}
 \Crel^{-1} \, \norm{\nabla(u - u_\ell)}{L^2(\Omega)}
 \le \eta_\ell \to 0
 \quad \text{as } \ell \to \infty.
\end{align}
\end{proposition}

\subsection{Fractional Laplacian}\label{section:fractional}

Let $\Omega \subset \R^d$ be a bounded Lipschitz domain, $d \ge 2$, and $0 < s < 1$. Given $f \in H^{-s}(\Omega)$, we consider the Dirichlet problem of the fractional Laplacian
\begin{align}\label{eq:fractional:strong}
 (-\Delta)^s u = f \text{ in } \Omega 
 \quad\text{subject to}\quad
  u = 0 \text{ in } \R^d \backslash \overline\Omega.
\end{align}
There are several different ways to define 
$(-\Delta)^s$, e.g., in terms of the Fourier transformation~\cite{MR3893441}, via semi-group theory~\cite{MR3613319}, or as Dirichlet-to-Neumann map of a half-space extension problem~\cite{MR2354493}. For the latter, a convenient representation of the fractional Laplacian is given in terms of a principal value integral
\begin{align}
\big( (-\Delta)^s u \big)(x) := C(d,s) \,\, {\rm p.v.}\!\!\! \int_{\R^d} \frac{u(x) - u(y)}{|x-y|^{d+2s}} \d{y}
 \text{ \ with \ } 
 C(d,s) := -2^{2s} \, \frac{\Gamma(s+d/2)}{\pi^{d/2}\Gamma(-s)},
\end{align}
where $\Gamma(\cdot)$ denotes the Gamma function. According to~\cite[Theorem~1.1]{MR3613319}, the weak formulation of~\eqref{eq:fractional:strong} reads: Find $u \in \XX := \widetilde H^s(\Omega)$ such that
\begin{align}\label{eq:fractional:weak}
 a(u,v) := \frac{C(d,s)}{2} \int\!\!\!\!\int_{\R^d\times\R^d} \!\!\!\!\!\!\!\!\! \frac{[u(x)-u(y)][v(x)-v(y)]}{|x-y|^{d+2s}} \d{x}\d{y}
 = \int_\Omega fv \d{x}
 \quad \text{for all } v \in \XX.
\end{align}
The Lax--Milgram lemma proves existence and uniqueness of $u \in \XX$. 

Following~\cite{MR3725761,fmp2019+}, we consider regular triangulations $\TT_\coarse$ of $\Omega$ into non-degenerate compact simplices and the corresponding first-order Courant finite element space $\XX_\coarse$ from~\eqref{eq:courant}.
Let $u_\coarse \in \XX_\coarse$ the corresponding Galerkin solution, i.e.,
\begin{align}
 a(u_\coarse, v_\coarse) = \int_\Omega f v_\coarse \d{x}
 \quad \text{for all } v_\coarse \in \XX_\coarse.
\end{align}
Under the additional regularity assumption $f \in L^2(\Omega)$, the local contributions of the error estimator from~\cite{fmp2019+} read
\begin{subequations}\label{eq:fractional:estimator}
\begin{align}
 \eta_\coarse(T) :=  \, \norm{h_\coarse^s \,[ f - (-\Delta)^s u_\coarse] }{L^2(T)}
 \quad \text{for all } T \in \TT_\coarse,
\end{align}
with the modified local mesh-width
\begin{align}
 h_\coarse^s|_T := \begin{cases}
 |T|^{s/2} & \text{for } 0 < s \le 1/2, \\
 |T|^{1/4} \dist(\cdot,\partial T)^{s-1/2} & \text{for } 1/2 \le s < 1.
 \end{cases}
\end{align}
\end{subequations}
According to~\cite[Theorem~2.3]{fmp2019+}, the error estimator is reliable 
\begin{align}\label{eq:reliable}
 \norm{u - u_\coarse}{\XX} \le \Crel \, \eta_\coarse.
\end{align}
Provided~\eqref{eq:union_of_sons} and~\eqref{eq:reduction:h}, stability~\eqref{eq:axiom:stability} and reduction~\eqref{eq:axiom:reduction} are proved in~\cite[Proposition~3.1]{fmp2019+} with $S(t) = \Cstab \,t$ and $R(t) = \Cred \, t^2$. The constants $\Crel, \Cstab > 0$ depend only on $\Omega$, $d$, $s$, and uniform $\gamma$-shape regularity~\eqref{eq:gamma-shape-regular}, while $\Cred > 0$ and $0 < \contract{red} < 1$ depend additionally on $\contract{ctr}$.
 Finally, {\sl a~priori} convergence~\eqref{assumption:limit_u} follows as in the seminal work~\cite{bv1984} (and essentially with the same arguments as in the previous section): The assumption~\eqref{eq:union_of_sons} on the mesh-refinement implies that refinement leads to nested spaces, i.e., Algorithm~\ref{alg:abstract_algorithm} leads to $\XX_\ell \subseteq \XX_{\ell+1}$ for all $\ell \in \N_0$. Therefore, $\XX_\infty := {\rm closure}\big(\bigcup_{\ell \in \N_0} \XX_\ell\big)$ is a closed subspace of $\XX$ and the Lax--Milgram lemma guarantees existence and uniqueness of $u_\infty \in \XX_\infty$ such that
\begin{align}
 a(u_\infty, v_\infty) = \int_\Omega fv_\infty \d{x}
 \quad \text{for all } v_\infty \in \XX_\infty.
\end{align}
With the Galerkin orthogonality and the resulting C\'ea lemma, it follows that
\begin{align}
 \norm{u_\infty - u_\ell}{\XX} \lesssim \min_{v_\ell \in \XX_\ell} \norm{u_\infty - v_\ell}{\XX}
 \to 0 \quad \text{as } \ell \to \infty.
\end{align}
Overall, we thus get the following plain convergence result, where we note that for D\"orfler marking~\eqref{eq:marking:doerfler} with sufficiently small marking parameter $0 < \theta \ll 1$, \cite[Theorem~2.6]{fmp2019+} proves even rate-optimal convergence.

\begin{proposition}
As long as the mesh-refinement strategy guarantees regular simplicial triangulations satisfying~\eqref{eq:union_of_sons}, \eqref{eq:reduction:h}, and~\eqref{eq:gamma-shape-regular}
 and as long as the marking strategy satisfies~\eqref{eq:marking}, Algorithm~\ref{alg:abstract_algorithm} for the fractional Laplacian~\eqref{eq:fractional:weak} driven by the indicators~\eqref{eq:fractional:estimator} yields convergence
\begin{align}\tag*{\qed}
 \Crel^{-1} \, \norm{u - u_\ell}{\H^s(\Omega)}
 \le \eta_\ell \to 0
 \quad \text{as } \ell \to \infty.
\end{align}
\end{proposition}


\subsection{Weakly-singular integral equations}

Let $\Omega \subset \R^d$ with $d = 2,3$ be a Lipschitz domain with compact boundary $\partial\Omega$ such that $\diam(\Omega)<1$ if $d=2$.
On a (relatively) open subset $\Gamma \subseteq \partial\Omega$ and given $f \in H^{1/2}(\Gamma)$, the weakly-singular integral equation
\begin{align}\label{eq:weaksing:strong}
(Vu)(x) := \int_\Gamma G(x-y) u(y) \d{y} = f(x)
 \quad \text{for all } x \in \Gamma
\end{align}
seeks the unknown integral density $u \in \XX := \H^{-1/2}(\Omega)$. Here, $G(\cdot)$ is the fundamental solution of the Laplacian, i.e., $G(z) = -\frac{1}{2\pi} \log|z|$ for $d = 2$ resp.\ $G(z) = \frac{1}{4\pi} |z|^{-1}$ for $d = 3$. 
We note that for $\Gamma = \partial\Omega$, \eqref{eq:weaksing:strong} is equivalent to the Dirichlet problem
\begin{align*}
 -\Delta U = 0 \text{ in } \Omega
 \quad \text{subject to } U = f \text{ on } \partial\Omega,
\end{align*}
supplemented by the appropriate radiation condition if $\Omega$ is unbounded; see~\cite{mclean}.
The weak formulation reads
\begin{align}\label{eq:weaksing:weak}
 a(u, v) := \int\!\!\!\!\int_{\Gamma\times\Gamma} G(x-y) u(y) v(x) \d{y} \d{x} 
 = \int_\Gamma fv \d{x}
 \quad \text{for all } v \in \XX,
\end{align}
and the Lax--Milgram lemma yields existence und uniqueness of the solution $u \in \XX$. 

For a fixed polynomial degree $p \ge 0$ and a regular triangulation $\TT_\coarse$ of $\Gamma$ into non-degenerate compact surface simplices, we consider standard boundary element spaces $\XX_\coarse = \PP^p(\TT_\coarse)$ consisting of $\TT_\coarse$-piecewise polynomials of degree $\le p$ (w.r.t.\ the boundary parametrization). Let $u_\coarse \in \XX_\coarse$ be the corresponding Galerkin solution, i.e.,
\begin{align}
 a(u_\coarse, v_\coarse) = \int_\Gamma f v_\coarse \d{x}
 \quad \text{for all } v_\coarse \in \XX_\coarse.
\end{align}
According to the seminal work~\cite{cms2001} and under additional regularity $f \in H^1(\Gamma)$, the local contributions of the residual error estimator read
\begin{align}\label{eq:weaksing:estimator}
 \eta_\coarse(T) := h_T^{1/2} \, \norm{\nabla_\Gamma(f - Vu_\coarse)}{L^2(T)}
 \quad\text{with}\quad
 h_T := |T|^{1/(d-1)},
\end{align}
where $\nabla_\Gamma(\cdot)$ is the surface gradient and $|\cdot|$ is the surface measure. While~\cite{cms2001} proves reliability~\eqref{eq:reliable}, stability~\eqref{eq:axiom:stability} and reduction~\eqref{eq:axiom:reduction} have first been proved in~\cite{fkmp2013, gantumur2013} with $S(t) = \Cstab \,t$ and $R(t) = \Cred \, t^2$ provided that \eqref{eq:union_of_sons} and~\eqref{eq:reduction:h} are satisfied.
The constants $\Crel, \Cstab > 0$ depend only on $\Gamma$, $d$, $p$, and uniform $\gamma$-shape regularity of the admissible meshes $\TT_\coarse \in \T$, i.e., 
\begin{subequations}\label{eq:Gamma:gamma-shape-regular} 
\begin{align}
 \gamma &:= \sup_{\TT_\coarse \in \T} \max_{\substack{T,T' \in \TT_\coarse\\ T\cap T'\neq\emptyset}} \frac{|T'|}{|T|} < \infty\quad\quad\,\,\, \text{if } d=2, 
\intertext{resp.}
 \gamma &:= \sup_{\TT_\coarse \in \T} \max_{T \in \TT_\coarse} \frac{{\rm diam}(T)}{|T|^{1/(d-1)}} < \infty\quad\text{if } d=3.
\end{align}
\end{subequations}
while $\Cred > 0$ and $0 < \contract{red} < 1$ depend additionally on $\contract{ctr}$.
The {\sl a~priori} convergence follows as in the previous section. Overall, we thus get the subsequent plain convergence result, where we note that for D\"orfler marking~\eqref{eq:marking:doerfler} with sufficiently small marking parameter $0 < \theta \ll 1$, \cite{fkmp2013, gantumur2013} proves even rate-optimal convergence. We also refer to our recent work~\cite{gp2020+}, which proves well-posedness of the residual estimator~\eqref{eq:weaksing:estimator} together with reliability, stability, and reduction for a large class of second-order elliptic PDEs with constant coefficients and the related weakly-singular integral operator $V$ as well as general mesh-refinement strategies.

\begin{proposition}
As long as the mesh-refinement strategy guarantees 
regular simplicial triangulations satisfying~\eqref{eq:union_of_sons}, \eqref{eq:reduction:h}, and~\eqref{eq:Gamma:gamma-shape-regular},
and as long as the marking strategy satisfies~\eqref{eq:marking}, Algorithm~\ref{alg:abstract_algorithm} for the weakly-singular integral equation~\eqref{eq:weaksing:weak} driven by the indicators~\eqref{eq:weaksing:estimator} yields convergence
\begin{align}
 \Crel^{-1} \, \norm{u - u_\ell}{\H^{-1/2}(\Gamma)}
 \le \eta_\ell \to 0
 \quad \text{as } \ell \to \infty.
\end{align}
\end{proposition}

\subsection{Hyper-singular integral equations}

Let $\Omega \subset \R^d$ with $d = 2,3$ be a Lipschitz domain with compact boundary $\partial\Omega$.
On a (relatively) open and connected subset $\Gamma \subseteq \partial\Omega$ and given $f \in H^{-1/2}(\Gamma)$ with $\int_\Gamma f \d{x}=1$ 
in case of $\Gamma=\partial\Omega$ the hyper-singular integral equation
\begin{align}\label{eq:hypsing:strong}
 (Wu)(x) := \text{p.v.} \int_\Gamma \frac{\partial_x}{\partial\nu(x)}\frac{\partial_y}{\partial\nu(y)}G(x-y) u(y) \d{y} = f(x)
 \quad \text{for all } x \in \Gamma
\end{align}
seeks the unknown integral density $u \in \XX := \set{v\in H^{1/2}(\Omega)}{\int_\Gamma v\d{x}=0}$ if $\Gamma=\partial\Omega$ resp.\ $u\in\XX:=\H^{1/2}(\Gamma)$ if $\Gamma\subsetneqq\partial\Omega$. Here, $\nu(\cdot)$ denotes the exterior normal vector and $G(\cdot)$ is again the fundamental solution of the Laplacian, i.e., $G(z) = -\frac{1}{2\pi} \log|z|$ for $d = 2$ resp.\ $G(z) = \frac{1}{4\pi} |z|^{-1}$ for $d = 3$. 
We note that for $\Gamma = \partial\Omega$, \eqref{eq:hypsing:strong}  is equivalent to the Neumann problem
\begin{align*}
 -\Delta U = 0 \text{ in } \Omega
 \quad \text{subject to}\quad \frac{\partial U}{\partial\nu}  = f \text{ on } \partial\Omega,
\end{align*}
supplemented by the appropriate radiation condition if $\Omega$ is unbounded; see~\cite{mclean}.
The weak formulation reads
\begin{align}\label{eq:hypsing:weak}
 a(u, v) := \int\!\!\!\!\int_{\Gamma\times\Gamma}  G(x-y) \, {\rm curl}_\Gamma u(y) \, {\rm curl}_\Gamma v(x) \d{y} \d{x} 
 = \int_\Gamma fv \d{x}
 \quad \text{for all } v \in \XX,
\end{align}
where ${\rm curl}_\Gamma(\cdot)$ denotes the surface curl (resp.\ the arclength derivative for $d = 2$); see~\cite{mclean}. The 
Lax--Milgram lemma yields existence und uniqueness of the solution $u \in \XX$. 

For a fixed polynomial degree $p \ge 1$ and a regular triangulation $\TT_\coarse$ of $\Gamma$ into non-degenerate compact surface simplices, we consider standard boundary element spaces $\XX_\coarse = \mathcal{S}^p(\TT_\coarse)$ consisting of globally continuous $\TT_\coarse$-piecewise polynomials of degree $\le p$ (w.r.t.\ the boundary parametrization). Let $u_\coarse \in \XX_\coarse$ be  the corresponding Galerkin solution, i.e.,
\begin{align}
 a(u_\coarse, v_\coarse) = \int_\Gamma f v_\coarse \d{x}
 \quad \text{for all } v_\coarse \in \XX_\coarse.
\end{align}
According to the seminal work~\cite{cmps04} and under additional regularity $f \in L^2(\Gamma)$, the local contributions of the residual error estimator read
\begin{align}\label{eq:hypsing:estimator}
 \eta_\coarse(T) := h_T^{1/2} \, \norm{f - Wu_\coarse}{L^2(T)}
 \quad\text{with}\quad
 h_T := |T|^{1/(d-1)},
\end{align}
where $|\cdot|$ is the surface measure. While~\cite{cmps04} proves reliability~\eqref{eq:reliable}, stability~\eqref{eq:axiom:stability} and reduction~\eqref{eq:axiom:reduction} have first been proved in~\cite{gantumur2013,ffkmp15} with $S(t) = \Cstab \,t$ and $R(t) = \Cred \, t^2$ provided that \eqref{eq:union_of_sons} and~\eqref{eq:reduction:h} are satisfied.
The constants $\Crel, \Cstab > 0$ depend only on $\Gamma$, $d$, $p$, and uniform $\gamma$-shape regularity~\eqref{eq:Gamma:gamma-shape-regular}  of the admissible meshes $\TT_\coarse \in \T$, 
while $\Cred > 0$ and $0 < \contract{red} < 1$ depend additionally on $\contract{ctr}$.
The {\sl a~priori} convergence follows as in Section~\ref{section:fractional}. Overall, we thus get the subsequent plain convergence result, where we note that for D\"orfler marking~\eqref{eq:marking:doerfler} with sufficiently small marking parameter $0 < \theta \ll 1$, \cite{gantumur2013,ffkmp15} proves even rate-optimal convergence.

\begin{proposition}
As long as the mesh-refinement strategy guarantees regular simplicial triangulations satisfying~\eqref{eq:union_of_sons}, \eqref{eq:reduction:h}, and~\eqref{eq:Gamma:gamma-shape-regular}, and as long as the marking strategy satisfies~\eqref{eq:marking}, Algorithm~\ref{alg:abstract_algorithm} for the hyper-singular integral equation~\eqref{eq:hypsing:weak} driven by the indicators~\eqref{eq:hypsing:estimator} yields convergence
\begin{align}
 \Crel^{-1} \, \norm{u - u_\ell}{\H^{1/2}(\Gamma)}
 \le \eta_\ell \to 0
 \quad \text{as } \ell \to \infty.
\end{align}
\end{proposition}

\subsection{Nonlinear interface problems}

Let $\Omega \subset \R^d$ with $d = 2,3$ be a bounded Lipschitz domain with compact boundary $\Gamma:=\partial\Omega$ and exterior domain $\Omega^{\rm ext}:=\R^d\setminus\overline\Omega$ such that $\diam(\Omega)<1$ if $d=2$. 
Further, let $A:\R^d\to\R^d$ be a Lipschitz continuous and strongly monotone coefficient function in the sense that there exist constants $C_{\rm lip},C_{\rm mon}>0$ such that 
\begin{align}
 |Ax-Ay|&\le C_{\rm lip} |x-y| \quad\text{for all }x,y\in\R^d,
 \\
 C_{\rm mon}\norm{\nabla u - \nabla v}{L^2(\Omega)}^2 &\le \int_\Omega (A\nabla u -A \nabla v)\cdot (\nabla u - \nabla v) \d{x}
 \quad\text{for all }u,v\in H^1(\Omega).
\end{align}
For given data $f\in L^2(\Omega)$, $u_D\in H^{1/2}(\Gamma)$, and $\phi_N\in H^{-1/2}(\Gamma)$ with additional compatibility condition 
\begin{align}
\int_\Omega f \d{x} + \int_\Gamma \phi_N \d{x} = 0
\quad \text{in case of $d=2$},
\end{align} 
we consider the nonlinear interface problem 
\begin{subequations}\label{eq:fembem:strong}
\begin{align}
 -\div(A\nabla u) &= f \quad\text{in }\Omega,\\
 -\Delta u^{\rm ext} &= 0 \quad\text{in }\Omega^{\rm ext},\\
 u-u^{\rm ext} &= u_D\quad \text{on }\Gamma,\\
 (A\nabla u - \nabla u^{\rm ext})\cdot \nu &= \phi_N \quad \text{on }\Gamma,\\
 u^{\rm ext} &= \mathcal{O}(|x|^{-1}) \quad \text{as }|x|\to \infty.
\end{align}
\end{subequations}
We seek for a weak solution $(u,u^{\rm ext})\in H^1(\Omega)\times H_{\rm loc}^1(\Omega^{\rm ext})$, where $H_{\rm loc}^1(\Omega^{\rm ext}) = \set{v}{v\in H^1(\omega) \text{ for all open and bounded }\omega\subseteq\Omega^{\rm ext}}$. 
There are different ways to equivalently reformulate \eqref{eq:fembem:strong} as FEM-BEM coupling. 
To ease presentation, we restrict ourselves to the Bielak--MacCamy coupling~\cite{bm83}, but we stress that Proposition~\ref{prop:fembem:convergence} holds accordingly for  the Johnson--N\'ed\'elec coupling~\cite{jn80} as well as Costabel's symmetric coupling~\cite{costabel88}; see  \cite{affkmp13} for details. 
Recalling the single-layer operator $V$ from~\eqref{eq:weaksing:strong} and defining the adjoint double layer operator 
\begin{align}
 (K'\phi)(x):=\int_\Gamma \frac{\partial_x}{\partial\nu(x)} G(x,y) \phi(y)  \d{y} \quad\text{for all }\phi\in H^{-1/2}(\Gamma) \text{ and all }  x\in\Gamma, 
\end{align}
the variational formulation resulting from the Bielak--MacCamy coupling seeks some ${\bf u} = (u,\phi)\in\XX:= H^1(\Omega)\times H^{-1/2}(\Gamma)$ such that 
\begin{align}\label{eq:fembem:weak}
 \begin{split}
 a ({\bf u},{\bf v}) &:= 
 \int_\Omega (A\nabla u) \cdot \nabla v\d{x} + \int_\Gamma \big((1/2 - K') \phi\big)  v \d{x}  + \int_\Gamma (u - V\phi)\psi \d{x}
 \\
 &=\int_\Omega fv\d{x} + \int_\Gamma\phi_N v \d{x} - \int_\Gamma u_D\psi\d{x} =: F({\bf v}) 
 \quad \text{for all } {\bf v} = (v,\psi)\in\XX.
 \end{split}
\end{align}
According to \cite{affkmp13}, \eqref{eq:fembem:weak} is uniquely solvable provided that $C_{\rm mon}>1/4$.

For a regular triangulation $\TT_\coarse$ of $\Omega$ into non-degenerate compact simplices and the induced regular triangulation $\EE_\coarse$ of $\Gamma$ into non-degenerate compact surface simplices, we consider globally continuous $\TT_\coarse$-piecewise affine functions $\mathcal{S}^1(\TT_\coarse)$  to discretize $H^1(\Omega)$, and 
 $\EE_\coarse$-piecewise constant functions $\mathcal{P}^0(\EE_\coarse)$ to discretize $H^{-1/2}(\Gamma)$,
 i.e., $\XX_\coarse=\SS^1(\TT_\coarse)\times\PP^0(\EE_\coarse)$.
  Let ${\bf u}_\coarse \in \XX_\coarse$ be  the corresponding Galerkin solution, i.e.,
\begin{align}
 a({\bf u}_\coarse, {\bf v}_\coarse) =F({\bf v}_\coarse)
 \quad \text{for all } {\bf v}_\coarse \in \XX_\coarse.
\end{align}
According to~\cite{affkmp13}, the local contributions of the residual error estimator read
\begin{align}\label{eq:fembem:estimator}
 \begin{split}
 &\eta_\coarse(T)^2 := 
 h_T^2 \norm{f}{L^2(T)}^2 + h_T \Big(\norm{[(A\nabla u_\coarse)\cdot \nu]}{L^2(\partial T\cap \Omega)}^2
 \\ 
 &+ \norm{\phi_N+(K'-1/2)\phi_\coarse-(A\nabla u_\coarse)\cdot\nu}{L^2(\partial T\cap\Gamma)}^2 + \norm{\nabla_\Gamma (u_\coarse-u_D-V\phi_\coarse)}{L^2(\partial T\cap\Gamma)}^2\Big)
 \end{split}
\end{align}
with the surface gradient $\nabla_\Gamma(\cdot)$ and the mesh-size $h_T:=|T|^{1/d}$ for all $T\in\TT_\coarse$.
Indeed, \cite{affkmp13} proves reliability~\eqref{eq:reliable}, stability~\eqref{eq:axiom:stability}, and reduction~\eqref{eq:axiom:reduction} (where the terms $u_\coarse,u_\fine$ are replaced by ${\bf u}_\coarse,{\bf u}_\fine$)  with $S(t) = \Cstab \,t$ and $R(t) = \Cred \, t^2$ provided that \eqref{eq:union_of_sons} and~\eqref{eq:reduction:h} are satisfied. 
The constants $\Crel, \Cstab > 0$ depend only on $\Gamma$, $d$, $C_{\rm lip}$, and uniform $\gamma$-shape regularity~\eqref{eq:gamma-shape-regular}  of the admissible meshes $\TT_\coarse \in \T$, 
while $\Cred > 0$ and $0 < \contract{red} < 1$ depend additionally on $\contract{ctr}$.
The {\sl a~priori} convergence follows as in Section~\ref{section:fractional}, where the required C\'ea lemma is given in~\cite[Corollary~12]{affkmp13}. Overall, we thus get the subsequent plain convergence result.

\begin{proposition}\label{prop:fembem:convergence}
As long as the mesh-refinement strategy guarantees regular simplicial triangulations satisfying~\eqref{eq:union_of_sons}, \eqref{eq:reduction:h}, and~\eqref{eq:gamma-shape-regular}, and as long as the marking strategy satisfies~\eqref{eq:marking}, Algorithm~\ref{alg:abstract_algorithm} for the Bielak--MacCamy coupling~\eqref{eq:fembem:weak} driven by the indicators~\eqref{eq:fembem:estimator} yields convergence
\begin{align}
 \Crel^{-1} \, \norm{{\bf u} - {\bf u}_\ell}{H^1(\Omega)\times H^{-1/2}(\Gamma)}
 \le \eta_\ell \to 0
 \quad \text{as } \ell \to \infty.
\end{align}
\end{proposition}


\bibliographystyle{alpha}
\bibliography{literature}

\end{document}